\documentclass{elsart}
\usepackage[english]{babel}
\usepackage{amsmath,amssymb,enumerate,bbm,xcolor}

\newenvironment{proof}{ {\it Proof.} }{\hfill{\it{QED}}\medskip}
\catcode`\<=\active \def<{
\fontencoding{T1}\selectfont\symbol{60}\fontencoding{\encodingdefault}}
\catcode`\>=\active \def>{
\fontencoding{T1}\selectfont\symbol{62}\fontencoding{\encodingdefault}}
\catcode`\|=\active \def|{
\fontencoding{T1}\selectfont\symbol{124}\fontencoding{\encodingdefault}}
\newcommand{\assign}{:=}
\newcommand{\mathd}{\mathrm{d}}
\newcommand{\nocomma}{}
\newcommand{\noplus}{}
\newcommand{\nosymbol}{}
\newcommand{\tmem}[1]{{\em #1\/}}
\newcommand{\tmmathbf}[1]{\ensuremath{\boldsymbol{#1}}}
\newcommand{\tmname}[1]{\textsc{#1}}
\newcommand{\tmop}[1]{\ensuremath{\operatorname{#1}}}
\newcommand{\tmstrong}[1]{\textbf{#1}}
\newcommand{\tmtextit}[1]{{\itshape{#1}}}
\newenvironment{enumeratenumeric}{\begin{enumerate}[1.] }{\end{enumerate}}
\newtheorem{lemma}{Lemma}
\newtheorem{theorem}{Theorem}

\begin{document}

\begin{frontmatter}
  \title{Sharp $L^{p}$ estimates for discrete second order Riesz transforms}
  
  \author[IMT]{Komla Domelevo}
  \author[IMT]{Stefanie Petermichl\thanksref{SP}\thanksref{fn1}}
  \ead{stefanie.petermichl@gmail.com}
  \ead[url]{http://math.univ-toulouse.fr/{$\tilde{\hspace{0.5em}}$}petermic}

  \thanks[SP]{Research supported in part by ANR-12-BS01-0013-02. The author is a member of IUF}
  \thanks[fn1]{Correponding author, Tel:+33 5 61 55 76 59, Fax: +33 5 61 55 83 85}
  \address[IMT]{Institut de Math{\'e}matiques de Toulouse, Universit{\'e} Paul Sabatier, Toulouse, France}
  
  \begin{abstract}
    We show that multipliers of second order Riesz transforms on products of
    discrete abelian groups enjoy the $L^{p}  $ estimate $p^{\ast} -1$, where
    $p^{\ast} = \max \{ p,q \}$ and $p$ and $q$ are conjugate exponents. This
    estimate is sharp if one considers all multipliers of the form $\sum_{i}
    \sigma_{i} R_{i} R^{\ast}_{i} \nocomma$ with $| \sigma_{i} | \leqslant 1$
    and infinite groups. In the real valued case, we obtain better sharp 
    estimates for some specific multipliers, such as $\sum_{i} \sigma_{i}
    R_{i} R^{\ast}_{i} \nocomma$ with $0 \leqslant \sigma_{i} \leqslant 1$.
    These are the first known precise $L^{p} $ estimates for discrete
    Calder{\'o}n-Zygmund operators.
  \end{abstract}


\end{frontmatter}

\section{Introduction}

Sharp $L^{p}$ estimates for norms of classical Calder{\'o}n-Zygmund operators
have been intriguing for a very long time. Only in rare cases, such optimal
constants are known and they all rely in one way or another on the existence
of a specific function of several variables that governs the proof. It is a
result by Pichorides {\cite{Pic1972}}, that determined the exact $L^{p}$ norm
of the Hilbert transform in the disk. There is also a very streamlined proof
by Ess{\'e}n {\cite{Ess1984}} on the same subject. That this estimate carries
over to $\mathbbm{R}^{N}$ and the Riesz transforms is an observation by
Iwaniec and Martin {\cite{IwaMar1996a}}. A probabilistic counterpart in the
language of orthogonal martingales was deduced by Ba{\~n}uelos and Wang in
{\cite{BanWan1995}}, using the idea of Ess{\'e}n in a probabilistic setting. 

The only other sharp $L^{p}$ estimates for classical Calder{\'o}n-Zygmund
operators we are aware of are those closely related to the real part of the
Beurling-Ahlfors operator, $R^{2}_{1} -R^{2}_{2}$, where $R_{i}$ are the two
Riesz transforms in $\mathbbm{R}^{2}$. The first estimate of this type is due
to Nazarov and Volberg {\cite{VolNaz2004a}} and was used for a partial result
towards the famous $p^{\ast} -1$ conjecture, the $L^{p}$ norm of the
Beurling-Ahlfors transform. It was not clear until the work of Geiss,
Montgomery-Smith and Saksman {\cite{GeiMonSak2010a}}, that the Nazarov-Volberg
estimate for the second order Riesz transform was infact sharp. In contrast to
the estimate for the Hilbert transform, in this case an underlying
probabilistic estimate - Burkholder's famous theorem on differentially
subordinate martingales - was known first and could be used to obtain the
optimal estimate for the real part of the Beurling-Ahlfors operator and its
appropriate generalisations to $\mathbbm{R}^{N}$.

It has been observed, that these $L^{p}$ estimates can be obtained without any
loss for the case of compact Lie groups, see Arcozzi {\cite{Arc1998a}} and
Ba{\~n}uelos and Baudoin {\cite{BanBau2013}}.

The situation changes, when the continuity assumption is lost. Optimal 
norm estimates for the discrete Hilbert transform on the integers are a long standing 
open question. Infact, until now, there were no known sharp $L^{p}$ estimates for any 
Calder{\'o}n-Zygmund operator when entering discrete abelian groups such as
$\mathbbm{Z}$ or cyclic groups $\mathbbm{Z}/m\mathbbm{Z}$. In the present paper, we provide such optimal estimates for discrete second order Riesz transforms. 

The definition of
derivatives and thus the Laplacian respects the discrete setting. For each
direction $i$, there are two choices of Riesz transforms, that depend on
whether right or left handed derivatives are used:
\begin{eqnarray*}
  R^{i}_{+} \circ \sqrt{- \Delta} = \partial^{i}_{+} \hspace{1em} \tmop{and}
  \hspace{1em} R^{i}_{-} \circ \sqrt{- \Delta} = \partial_{-}^{i} .
\end{eqnarray*}
One often sees $R_{i}  $written for $R^{i}_{+}$ and $R^{\ast}_{i}$ for
$-R^{i}_{-}$. The discrete derivatives that come into play induce difficulties
that are due to their non-local nature. To underline the artifacts one may
encounter, we briefly discuss a related subject, dimensional behavior of Riesz
vectors in a variety of different situations.

F. Lust-Piquard showed in {\cite{Lus1998a}} and {\cite{Lus2004a}} that square
functions of Riesz transforms on products of discrete abelian groups of a
single generator $\mathbbm{Z},\mathbbm{Z}/k\mathbbm{Z}$ for k=2,3,4 enjoy
dimension-free $L^{p}  $estimates when $p \geqslant 2$ but suffer from a
dimension dependence when $1<p<2$. She used non-commutative methods to obtain
this nice positive result for $p \geqslant 2$. She provided an example
originating on the hyper cube $( \mathbbm{Z}/2\mathbbm{Z} )^{N}  $,
demonstrating dimensional growth when $p<2$.

This difficulty of dimensional dependence for some $p$ in the discrete case
is in a sharp contrast to numerous, both interesting and difficult,
dimension-free estimates in some very general continuous settings. We refer
the reader to the result by Stein {\cite{Ste1983a}} in the classical case.
Probabilistic methods, applied by P.A. Meyer {\cite{Mey1984a}} lead to the
first proof of this theorem in the Gaussian setting, whereas Pisier found an
analytic proof in {\cite{Pis1988a}}. We refer the reader to more general
settings on compact Lie groups {\cite{Arc1998a}}, Heisenberg groups
{\cite{CouMulZie1996a}} and the best to date estimate for Riesz vectors on
Riemannian manifolds with a condition on curvature {\cite{CarDra2013a}}.

\section{Definitions and main results}

Let \ $G$ be an additive discrete abelian group generated by $e$ and $f:G
\rightarrow \mathbbm{C}$. Its right and left hand derivatives are
\begin{eqnarray*}
  \partial_{+} f ( n ) =f ( n+e ) -f ( n ) \hspace{1em} \tmop{and}
  \hspace{1em} \partial_{-} f ( n ) =f ( n ) -f ( n-e ) .
\end{eqnarray*}
The reader may keep in mind that $G \in \{
\mathbbm{Z},\mathbbm{Z}/m\mathbbm{Z};m \geqslant 2 \}$. We then define partial
derivatives $\partial^{i}_{\pm}$ accordingly on products of such groups. Here,
we denote by $e_{i}$ the $i$th unit element in $G^{N}$ and
\begin{eqnarray*}
  \partial^{i}_{+} f ( n ) =f ( n+e_{i} ) -f ( n )   \hspace{1em} \tmop{and}
  \hspace{1em} \partial^{i}_{-} f ( n ) =f ( n ) -f ( n-e_{i} ) .
\end{eqnarray*}
A priori, we consider $G^{N}$, but it will be clear that our methods apply to
mixed cases $\otimes^{N}_{i=1} G_{i}$ where $G_{i}$ are not necessarily the
same discrete groups. The discrete Laplacian becomes
\begin{eqnarray*}
  \Delta = \sum_{i=1}^{N} \partial^{i}_{+} \partial^{i}_{-} = \sum_{i=1}^{N}
  \partial^{i}_{-} \partial^{i}_{+} .
\end{eqnarray*}
For each direction $i$ there are two choices of Riesz transforms, defined in a
standard manner:
\begin{eqnarray*}
  R^{i}_{+} \circ \sqrt{- \Delta} = \partial^{i}_{+} \hspace{1em} \tmop{and}
  \hspace{1em} R^{i}_{-} \circ \sqrt{- \Delta} = \partial_{-}^{i} .
\end{eqnarray*}
In this text, we are concerned with second order Riesz transforms defined as
\begin{eqnarray*}
  R^{2}_{i} \assign R^{i}_{+} R^{i}_{-} =-R_{i} R^{\ast}_{i}
\end{eqnarray*}
and combinations thereof, such as
\begin{eqnarray*}
  R_{\alpha}^{2} \assign \sum^{N}_{i=1} \alpha_{i} R^{2}_{i} ,
\end{eqnarray*}
where the $\alpha = \{ \alpha_{i} \} \in \mathbbm{C}^{N} : | \alpha_{i} |
\leqslant 1$. When $p$ and $q$ are conjugate exponents, let
\begin{eqnarray*}
  p^{\ast} = \max \{ p,q \} = \left\{ \begin{array}{ll}
    p & 2 \leqslant p\\
    q= \tfrac{p}{p-1}   & 1<p \leqslant 2
  \end{array} \right. .
\end{eqnarray*}
So \ $p^{\ast} -1= \max \left\{ p-1, \tfrac{1}{p-1} \right\} .$ The following
are our main results.

\begin{theorem}
  \label{T: p minus 1 estimate}$R^{2}_{\alpha} :L^{p} ( G^{N} ,\mathbbm{C} )
  \rightarrow L^{p} ( G^{N} ,\mathbbm{C} )$ enjoys the operator norm estimate
   $ \| R_{\alpha}^{2} \| \leqslant p^{\ast} -1.$
  The estimate above is sharp when $G=\mathbbm{Z}$ and $N \geqslant
  2$.{\tmname{}}
\end{theorem}

In some cases, better estimates are available in the real valued case.

\begin{theorem}
  \label{T: Choi constant estimate}$R^{2}_{\alpha} :L^{p} ( G^{N} ,\mathbbm{R}
  ) \rightarrow L^{p} ( G^{N} ,\mathbbm{R} )$ with $\alpha = \{ \alpha_{i}
  \}_{i=1, \ldots ,N} \in \mathbbm{R}^{N}$ enjoys the sharp norm estimate $\|
  R_{\alpha}^{2} \| \leqslant \mathfrak{C}_{\min   \alpha , \nocomma \max  
  \alpha ,p} \nocomma$, where these are the Choi constants. 
\end{theorem}

The Choi constants (see {\cite{Cho1992a}}) are not explicit, except
$\mathfrak{C}_{-1,1,p} =p^{\ast} -1$. However, about $\mathfrak{C}_{0,1,p}$ it
is known that
\begin{eqnarray*}
  \mathfrak{C}_{0,1,p} = \tfrac{p}{2} + \tfrac{1}{2} \log \left(
  \tfrac{1+e^{-2}}{2} \right) + \tfrac{\beta_{2}}{p} + \ldots .
\end{eqnarray*}
with $\beta_{2} = \log^{2} \left( \tfrac{1+e^{-2}}{2} \right) + \tfrac{1}{2}
\log \left( \tfrac{1+e^{-2}}{2} \right) -2 \left( \tfrac{e^{-2}}{1+e^{-2}}
\right)^{2} .$

The proof of our results uses Bellman functions and passes via embedding
theorems that are a \ somewhat stronger estimate than Theorem \ref{T: p minus
1 estimate} and Theorem \ref{T: Choi constant estimate}.

Abelian groups of a single generator are isomorphic to $\mathbbm{Z}$ or
$\mathbbm{Z}/m\mathbbm{Z}$ for some integer $m \geqslant 2.$ We are going to
make use of the Fourier transform on these groups. In the case of the infinite
group $\mathbbm{Z}$, the charater set is $\mathbbm{T}= \left[ - \frac{1}{2} ,
\frac{1}{2} \right] .$ We briefly review the Fourier Transform in
$\mathbbm{Z}^{N}$. We write $\mathbbm{T}^{N}  $for $\left[ - \frac{1}{2} ,
\frac{1}{2} \right]^{N}$. $\mathbbm{T}^{N}$ is the character group of
$\mathbbm{Z}^{N}$.
\begin{eqnarray*}
  \widehat{} : ( f:\mathbbm{Z}^{N} \rightarrow \mathbbm{C} ) \rightarrow (
  \hat{f} :\mathbbm{T}^{N} \rightarrow \mathbbm{C} ) ,f ( n ) \mapsto \hat{f}
  ( \xi ) = \sum_{\mathbbm{Z}^{N}} f ( n ) e^{-2 \pi i n \cdot \xi}
\end{eqnarray*}
with inversion formula
\begin{eqnarray*}
  f ( n ) = \int_{\mathbbm{T}^{N}} \hat{f} ( \xi ) e^{2 \pi i n \cdot \xi} d
  \xi .
\end{eqnarray*}
We observe that discrete derivatives are multiplier operators
\begin{eqnarray*}
  \widehat{\partial^{j}_{\pm}  f} ( \xi ) = \pm \hat{f} ( \xi ) \cdot ( e^{\pm
  2 \pi i \xi_{j}} -1 ) = \hat{f} ( \xi ) \cdot 2i e^{\pm \pi i \xi_{j}} \sin
  ( \pi \xi_{j} ) .
\end{eqnarray*}
For the discrete Laplacian there holds
\begin{eqnarray*}
  ( \Delta f ) ( n ) & = & \sum_{i} ( \partial^{i}_{-} \partial^{i}_{+} f ) (
  n ) = \sum_{i} \partial^{i}_{-} ( f ( n+e_{i} ) -f ( n ) )\\
  & = & \sum_{i} f ( n+e_{i} ) -2f ( n ) +f ( n-e_{i} ) .
\end{eqnarray*}
We obtain the multiplier of $\Delta$:
\begin{eqnarray*}
  \hat{\Delta} = \sum_{i} ( -2+e^{2 \pi i \xi_{i}} +e^{-2 \pi i \xi_{i}} ) =-4
  \sum_{i} \sin^{2} ( \pi \xi_{i} ) ,
\end{eqnarray*}
and the multiplier of the left and right handed $j$th Riesz transforms:
\begin{eqnarray*}
  \widehat{R^{j}_{\pm}} = \frac{2 i e^{\pm \pi i \xi_{j}} \sin ( \pi
  \xi_{j} )}{\sqrt{4 \sum_{i} \sin^{2} ( \pi \xi_{i} )}} .
\end{eqnarray*}
Second order Riesz transforms $R^{2}_{j} \assign R^{j}_{+} R_{-}^{j}  $ have
the following multipliers:
\begin{eqnarray*}
  \widehat{R^{2}_{j}} = \frac{-4 \sin^{2} ( \pi \xi_{j} )}{4 \sum_{i} \sin^{2}
  ( \pi \xi_{i} )} .
\end{eqnarray*}
In $( \mathbbm{Z}/m\mathbbm{Z} )^{N}$, not many changes are required. Recall
that a character $\chi$ of a locally compact abelian group $G$ is a
homomorphism $\chi :G \rightarrow S^{1}$. Let $e$ be a generator of our group
$G=\mathbbm{Z}/m\mathbbm{Z}$, so $m e=0$ and $n e \neq 0$ for $0<n<m$.
Therefore $\chi ( e )^{m} =1$ so that the character group of
$\mathbbm{Z}/m\mathbbm{Z}$ consists of the $m$ roots of unity. The character
set of $( \mathbbm{Z}/m\mathbbm{Z} )^{N}$ is therefore isomorphic to itself,
like for all finite abelian groups. The Fourier transform looks as follows:
\begin{eqnarray*}
  \hat{f} ( \xi ) = \sum_{( \mathbbm{Z}/m\mathbbm{Z} )^{N}} f ( n ) e^{-
  \frac{2 \pi i}{m} n \cdot \xi}
\end{eqnarray*}
as well as
\begin{eqnarray*}
  f ( n ) = \frac{1}{m^{N}} \sum_{( \mathbbm{Z}/m\mathbbm{Z} )^{N}} \hat{f} (
  \xi ) e^{\frac{2 \pi i}{m} n \cdot \xi} .
\end{eqnarray*}
These are special cases of the more general formulae $\hat{f} ( \chi ) =
\sum_{G} f ( g ) \bar{\chi} ( g )$ and $f ( x ) = \frac{1}{\sharp G} \sum_{G}
\hat{f} ( \chi ) \chi ( g )$. We observe the following about discrete
derivatives:
\begin{eqnarray*}
  \widehat{\partial^{i}_{\pm}} =2 i e^{\pm \frac{\pi i}{m} \xi_{i}} \sin
  \left( \frac{\pi}{m} \xi_{i} \right) \hspace{1em} \tmop{and} \hspace{1em}
  \hat{\Delta} = \widehat{\sum_{i} \partial^{i}_{-} \partial^{i}_{+}} =-4
  \sum_{i} \sin^{2} \left( \frac{\pi}{m} \xi_{i} \right) .
\end{eqnarray*}

The Riesz transforms are defined as above and we have for all directions $1
\leqslant j \leqslant N$,
\begin{eqnarray*}
  \widehat{R^{2}_{j}} = \frac{-4 \sin^{2} \left( \frac{\pi}{m} \xi_{j}
  \right)}{4 \sum_{i} \sin^{2} \left( \frac{\pi}{m} \xi_{i} \right)}.
\end{eqnarray*}
We will be using continuous in time heat extensions $\tilde{f} :G^{N} \times [
0, \infty ) \rightarrow \mathbbm{C}$ of functions $f:G^{N} \rightarrow
\mathbbm{C}$. Namely for all $t \geqslant 0$, $\tilde{f} ( t ) =e^{t \Delta}
f$. Using this notation, we are ready to state two further result, the
following bilinear estimates:

\begin{theorem}
  \label{T: bilinear embedding p minus 1}Let $f$ and $g$ be test functions on
  $G^{N}$ and $p$ and $q$ conjugate exponents. Then we have the estimate
  \begin{eqnarray}
    2 \sum_{i=1}^{N} \int^{\infty}_{0} \sum_{G^{N}} | \partial^{i}_{+}
    \tilde{f} ( n,t ) |   | \partial_{+}^{i} \tilde{g} ( n,t ) | \mathd  t
    \leqslant ( p^{\ast} -1 ) \| f \|_{p} \| g \|_{q} . \label{eq: bilinear
    embedding p minus 1}
  \end{eqnarray}
\end{theorem}

\begin{theorem}
  \label{T: bilinear embedding Choi}Let $f$ and $g$ be real valued test
  functions on $G^{N}$ and $p$ and $q$ conjugate exponents. Then we have the
  estimate
  \begin{eqnarray*}
    2 \sum_{i=1}^{N} \int^{\infty}_{0} \sum_{G^{N}} \left[ \partial^{i}_{+}
    \tilde{f} \, \partial^{i}_{+} \tilde{g} \right]_{\pm} \mathd t \leqslant
    \mathfrak{C}_{p} \| f \|_{p} \| g \|_{q} ,
  \end{eqnarray*}
  where $\left[ \, \cdot \, \right]_{\pm}$ denotes positive and negative
  parts.
\end{theorem}

\section{The $p^{\ast}  - 1$ estimates.}

We take advantage of an intimate connection between differentially subordinate
martingales and representation formulae for our type of singular integrals
using heat extensions. This idea was first used in {\cite{PetVol2002a}} \ in a
weighted context. This here is the first application to a discrete group.

\subsection{Representation formula}

We will be using continuous in time heat extensions $P_{t} f:G^{N} \times [ 0,
\infty ) \rightarrow \mathbbm{C}$ of a function $f:G^{N} \rightarrow
\mathbbm{C}$. Precisely, we set $P_{t} \assign e^{t \Delta}$, so that $( P_{t}
f ) ( t ) \assign e^{t \Delta} f$. Denoting for convenience $\tilde{f} ( t )
\assign P_{t} f$, we have that the function $\tilde{f}$ solves the
semi-continuous heat equation $\partial_{t} \tilde{f} - \Delta \tilde{f} =0$
with initial condition $\tilde{f} ( 0 ) =f$. The group structure allows us to
express the semi-group $P_{t}$ in terms of a convolution kernel involving the
elementary solution $K:G^{N} \times [ 0, \infty ) \rightarrow \mathbbm{C}$
solution to the semi-continuous heat equation with initial data $K ( n,0 )
\assign \delta_{0} ( n )$, where $\delta_{0} ( 0 ) =1$ and $\delta_{0} ( n )
=0$ for $n \neq 0$. We have $K \geqslant 0$ and for all $( n,t )$,
\begin{eqnarray}
  \tilde{f} ( n,t ) = ( K ( t ) \ast f ) ( n ) = \sum_{m \in G^{N}} K ( m,t )
  f ( 0,n-m ) , \sum_{m \in G^{N}} K ( m,t ) =1 \label{eq:
  elementary properties heat kernel}
\end{eqnarray}
We will note in general $K \left( \cdot , \cdot \, ;0,m \right)$ the
elementary solution with initial condition $K ( 0 ) \assign \delta_{0} \left(
\cdot \, -m \right)$ translated by $m \in G^{N}$, namely $K \left( n,t \, ;m,0
\right) =K ( n-m,t )$. It is well known that this kernel, defined on a
 regular graph such as $G^{N}$ decays exponentially with respect to its variable $n$, for
$G$ an infinite group.

Let $R^{2}_{i}$ be the square of the $i$th Riesz transform, as defined above.
Let $f$ and $g$ be test functions. Let us first note, that $R_{i}$ maps
constants to $0$, so we may assume that in the formulae below, $\hat{g} ( 0 )
=0.$

\begin{lemma}
  \label{L: representation formula}If $\hat{g} ( 0 ) =0$ then
  \begin{eqnarray}
    ( f,R^{2}_{j} g ) =-2 \int^{\infty}_{0} \sum_{G^{N}} \partial^{j}_{+}
    \tilde{f} ( n,t ) \overline{\partial_{+}^{j} \tilde{g} ( n,t )} \mathd  t
    \label{eq: representation formula}
  \end{eqnarray}
  and the sums and integrals that arise converge absolutely.
\end{lemma}

\begin{proof}
  In the infinite case, $\mathbbm{Z}$, we first notice that for $\xi \neq 0$,
  we have
  \begin{eqnarray*}
    2 \int^{\infty}_{0} e^{-8t \sum_{i} \sin^{2} ( \pi \xi_{i} )}
    \mathd t= 
    \left.- \frac{e^{-8t \sum_{i} \sin^{2} ( \pi \xi_{i} )}}{4 \sum_{i} \sin^{2} (
    \pi \xi_{i} )} \right\lvert ^{\infty}_{0} = \frac{1}{4 \sum_{i} \sin^{2} ( \pi
    \xi_{i} )}  .
  \end{eqnarray*}
  Now, we calculate
  \begin{eqnarray*}
    ( f,R^{2}_{j} g ) & = & \sum_{\mathbbm{Z}^{N}} f ( n ) \overline{R^{2}_{j}
    g ( n )} = \int_{\mathbbm{T}^{N}} \hat{f} ( \xi ) \frac{-4 \sin^{2} ( \pi
    \xi_{j} )}{4 \sum_{i} \sin^{2} ( \pi \xi_{i} )} \overline{\hat{g} ( \xi )}
    \mathd  A ( \xi )\\
    & = & -2 \int_{\mathbbm{T}^{N}} \int^{\infty}_{0} 4 \sin^{2} ( \pi
    \xi_{j} ) e^{-8t \sum_{i} \sin^{2} ( \pi \xi_{i} )} \widehat{ f} ( \xi )
    \overline{\hat{g} ( \xi )} \mathd t  \mathd  A ( \xi )\\
    & = & -2 \int_{\mathbbm{T}^{N}} \int^{\infty}_{0} 2i e^{\pi i \xi_{j}}
    \sin ( \pi \xi_{j} ) e^{-4t \sum_{i} \sin^{2} ( \pi \xi_{i} )} \widehat{
    f} ( \xi ) \times\\
    &  & \qquad \times \; \overline{2i e^{\pi i \xi_{j}} \sin ( \pi \xi_{j} ) e^{-4t
    \sum_{i} \sin^{2} ( \pi \xi_{i} )} \hat{g} ( \xi )} \mathd t  \mathd A (
    \xi )\\
    & = & -2 \int^{\infty}_{0} \sum_{\mathbbm{Z}^{N}} \partial^{j}_{+}
    \tilde{f} ( n,t ) \overline{\partial_{+}^{j} \tilde{g} ( n,t )} \mathd  t
  \end{eqnarray*}
  Here, $\tilde{f} ( n,t )$ denotes the semi-continuous heat extension of the
  function $f$.{\color{red} }

  In the cyclic case, there are not many changes. Observe that even in
  $\mathbbm{Z}/2\mathbbm{Z}$, where there is no `room' for three points in a
  row, the Laplacian is a well defined negative operator and the equation
  diffuses correctly. Also here, we may assume $\hat{g} ( 0 ) =0 \nosymbol .$
  As before, we have when $\xi \neq 0$
  \begin{eqnarray*}
    2 \int^{\infty}_{0} e^{-8t \sum_{i} \sin^{2} \left( \frac{\pi}{m} \xi_{i}
    \right)} \mathd t=  \left.- \frac{e^{-8t \sum_{i} \sin^{2} \left(
    \frac{\pi}{m} \xi_{i} \right)}}{4 \sum_{i} \sin^{2} \left( \frac{\pi}{m}
    \xi_{i} \right)} \right\lvert^{\infty}_{0} = \frac{1}{4 \sum_{i} \sin^{2}
    \left( \frac{\pi}{m} \xi_{i} \right)}  .
  \end{eqnarray*}
  We calculate in the same manner
  \begin{eqnarray*}
    ( f,R^{2}_{j} g ) & = & \sum_{( \mathbbm{Z}/m\mathbbm{Z} )^{N}} f ( n )
    \overline{R^{2}_{j} g ( n )} = \sum_{( \mathbbm{Z}/m\mathbbm{Z} )^{N}}
    \hat{f} ( \xi ) \frac{-4 \sin^{2} \left( \frac{\pi}{m} \xi_{j} \right)}{4
    \sum_{i} \sin^{2} \left( \frac{\pi}{m} \xi_{i} \right)} \overline{\hat{g}
    ( \xi )}\\
    & = & -2 \sum_{( \mathbbm{Z}/m\mathbbm{Z} )^{N}} \int^{\infty}_{0} 4
    \sin^{2} \left( \frac{\pi}{m} \xi_{j} \right) e^{-8t \sum_{i} \sin^{2}
    \left( \frac{\pi}{m} \xi_{i} \right)} \widehat{ f} ( \xi )
    \overline{\hat{g} ( \xi )} \mathd t \\
    & = & -2 \sum_{( \mathbbm{Z}/m\mathbbm{Z} )^{N}} \int^{\infty}_{0} 2i
    e^{\frac{\pi i}{m} \xi_{j}} \sin \left( \frac{\pi}{m} \xi_{j} \right)
    e^{-4t \sum_{i} \sin^{2} \left( \frac{\pi}{m} \xi_{i} \right)} \widehat{
    f} ( \xi ) \times\\
    &  &\qquad  \times \overline{2i e^{\frac{\pi i}{m} \xi_{j}} \sin \left(
    \frac{\pi}{m} \xi_{j} \right) e^{-4t \sum_{i} \sin^{2} \left(
    \frac{\pi}{m} \xi_{i} \right)} \hat{g} ( \xi )} \mathd t \\
    & = & -2 \int^{\infty}_{0} \sum_{( \mathbbm{Z}/m\mathbbm{Z} )^{N}}
    \partial^{j}_{+} \tilde{f} ( n,t ) \overline{\partial_{+}^{j} \tilde{g} (
    n,t )} \mathd  t
  \end{eqnarray*}
  This concludes the proof of Lemma \ref{L: representation formula}.
\end{proof}

We are going to use the Bellman function method to derive our estimates. In
order to do so, we are going to need a tool to control the terms we see on the
right hand side of our representation formulae. This is the content of the
next subsection.

\subsection{Bellman function: dissipation and size estimates}

In this section we derive the existence of a function with a certain convexity
condition. There is an explicit construction of (upto smoothness) the type of
function we need in Vasyunin-Volberg {\cite{VasVol2010a}} of an elaborate complexity,
as well as a new construction in Ba{\~n}uelos-Os{\c{e}}kowski {\cite{BanOse2014a}}. There is
also a much simpler explicit expression in {\cite{NazTre1996a}} at the cost of
a factor, which does bother us for our purposes. In our proof, we only need
the existence, not the explicit expression. This part is standard for readers
with some background in Bellman \ functions, however this argument is not
usually written carefully. We will also apply similar considerations to derive
a Bellman function that will give us our estimates in terms of the Choi
constant.

\begin{theorem}
  \label{T: Bellman function with smoothing}For any 1< $p< \infty$ we define
  \begin{eqnarray*}
    D_{p} \assign \{ \tmmathbf{v}= (
    \tmmathbf{F},\tmmathbf{G},\tmmathbf{f},\tmmathbf{g} ) \subset
    \mathbbm{R}^{+} \times \mathbbm{R}^{+} \times \mathbbm{C} \times
    \mathbbm{C}: | \tmmathbf{f} |^{p} <\tmmathbf{F}, | \tmmathbf{g} |^{q}
    <\tmmathbf{G} \} .
  \end{eqnarray*}
  Let $K$ be any compact subset of $D_{p}$ and let $\varepsilon$ be an
  arbitrary small positive number. Then there exists a function
  $B_{\varepsilon ,p,K} ( \tmmathbf{v} )$ that is infinitely differentiable in
  an $\varepsilon$-neighborhood of $K$ and such that
  \begin{eqnarray*}
    0 \leqslant B_{\varepsilon ,p,K} ( \tmmathbf{v} ) \leqslant ( 1+
    \varepsilon C_{K} ) ( p^{\ast} -1 ) \tmmathbf{F}^{1/p} \tmmathbf{G}^{1/q}
    ,
  \end{eqnarray*}
  and
  \begin{eqnarray}
    - \mathd_{\tmmathbf{v}}^{2} B_{\varepsilon ,p,K} ( \tmmathbf{v} )
    \geqslant 2 | \mathd \tmmathbf{f} |   | \mathd \tmmathbf{g} | . \label{eq:
    dissipation infinitesimal}
  \end{eqnarray}
\end{theorem}

Before we proceed with the proof of this theorem, we will require a lemma that
follows directly from Burkholder's famous $p^{\ast} -1$ estimate for
differentially subordinate martingales. Readers who do not wish to admit any
probability, can make use of an explicit expression mentioned above.

\begin{lemma}
  \label{L: Bellman function dyadic}For every $1<p< \infty$ there exists a
  function $B_{p}$ of four variables $\tmmathbf{v}= (
  \tmmathbf{F},\tmmathbf{G},\tmmathbf{f},\tmmathbf{g} )$ in the domain
  $\overline{D_{p}} = \{ ( \tmmathbf{F},\tmmathbf{G},\tmmathbf{f},\tmmathbf{g}
  ) \in \mathbbm{R}^{+} \times \mathbbm{R}^{+} \times \mathbbm{C} \times
  \mathbbm{C}: | \tmmathbf{f} |^{p} \leqslant \tmmathbf{F}, | \tmmathbf{g}
  |^{q} \leqslant \tmmathbf{G} \}$ so that
  \begin{eqnarray*}
    0 \leqslant B ( \tmmathbf{v} ) \leqslant ( p^{\ast} -1 )
    \tmmathbf{F}^{1/p} \tmmathbf{G}^{1/q} ,
  \end{eqnarray*}
  and
  \begin{eqnarray}
    B ( \tmmathbf{v} ) \geqslant \frac{1}{2} B ( \tmmathbf{v}_{+} ) +
    \frac{1}{2} B ( \tmmathbf{v}_{-} )   \noplus \noplus + \frac{1}{4} |
    \tmmathbf{f}_{+} -\tmmathbf{f}_{-} |   | \tmmathbf{g}_{+}
    -\tmmathbf{g}_{-} |   \label{eq: dissipation dyadic}
  \end{eqnarray}
  if $\tmmathbf{v}_{+} +\tmmathbf{v}_{-} =2\tmmathbf{v}$ and $\tmmathbf{v}$,
  $\tmmathbf{v}_{+}$, $\tmmathbf{v}_{-}$ are in the domain. Here $\|
  \cdot \|  $denotes the usual $\ell_{2}$ norm in $\mathbbm{C}.$ 
\end{lemma}

\begin{proof}
  The existence of this function follows from Burkholder's theorem directly in
  one of its simplest forms (see {\cite{Bur1984a}} and {\cite{Bur1988a}}):
  
  \begin{theorem}
    (Burkholder) Let $( \Omega ,\mathfrak{F},P )$ be a probability space with filtration 
    $\mathfrak{F}= ( \mathfrak{F}_{n} ) \nocomma_{n \in \mathbbm{N}}$. 
    Let $X_{}$ and $Y_{}$ be complex valued martingales with
    differential subordination $| Y_{0} ( \omega ) | \leqslant | X_{0} (
    \omega ) |$ and $| Y_{n} ( \omega ) -Y_{n-1} ( \omega ) | \leqslant |
    X_{n} ( \omega ) -X_{n-1} ( \omega ) |$ for almost all $\omega \in \Omega$
    . Then $\| Y \|_{L^{p}} \leqslant ( p^{\ast} -1 ) \| X \|_{L^{p}}$ where
    the $L^{p}$ norms are in the sense of martingales.
  \end{theorem}
  
  Now, if $\mathcal{D}$ is the dyadic grid on the real line, let
  $\mathfrak{F}_{n}  $contain all dyadic intervals of size at least $2^{-n} |
  J |$. For a test function $f$ supported on a dyadic interval $J$, its dyadic
  approximation gives rise to a martingale $\forall \omega \in J,X_{0} (
  \omega ) =\mathbbm{E}f$, with
  \begin{eqnarray*}
    \forall n>0, \hspace{1em} X_{n} ( \omega ) = \sum_{I \in \mathcal{D} ( J )
    , | I | >2^{-n} | J |} ( f,h_{I} ) h_{I} ( \omega ) .
  \end{eqnarray*}
  With the help of random multiplications $\sigma = \left\{ \sigma_{I} \in
  S^{1} ; \hspace{1em} I \in \mathcal{D} ( J ) \right\}  $ one defines the
  martingale transform $Y$ of $X$ as $\forall \omega \in J,Y_{0} ( w ) =X_{0}
  ( w )$ and
  \begin{eqnarray*}
    \forall n>0, \hspace{1em} Y_{n} ( \omega ) = \sum_{I \in \mathcal{D} ( J )
    , | I | >2^{-n} | J |} \sigma_{I} ( f,h_{I} ) h_{I} ( \omega ) .
  \end{eqnarray*}
  Here $X$ and $Y$ enjoy the differential subordination property. Consequently, 
  Burkholder's theorem asserts that $\| Y \|_{p} \leqslant ( p^{\ast} -1
  ) \| X \|_{p}$. This implies that the operator
  \begin{eqnarray*}
    T_{\sigma} :f \mapsto \sum_{I \in \mathcal{D} ( J )} \sigma_{I} ( f,h_{I}
    ) h_{I}
  \end{eqnarray*}
  has a uniform $L^{p}$ bound of at most $p^{\ast} -1 \nosymbol .$ So for test
  functions $f$, $g$ supported in $J$ we have the estimate
  \begin{eqnarray*}
    \sup_{\sigma} \frac{1}{| J |} | ( T_{\sigma} f,g ) | \leqslant ( p^{\ast}
    -1 ) \langle | f |^{p} \rangle_{J}^{1/p} \langle | g |^{q}
    \rangle_{J}^{1/q} .
  \end{eqnarray*}
  Expanding the orthonormal series on the left and choosing the worst $\sigma$
  yields the estimate
  \begin{eqnarray*}
    \frac{1}{| J |} \sum_{I \in \mathcal{D} ( J )} | ( f,h_{I} ) |   | (
    g,h_{I} ) |   & = & \frac{1}{4 | J |} \sum_{I \in \mathcal{D} ( J )} | I |
    | \langle f \rangle_{I_{+}} - \langle f \rangle_{I_{-}} |   | \langle g
    \rangle_{I_{+}} - \langle g \rangle_{I_{-}} |\\
    & \leqslant & ( p^{\ast} -1 ) \langle | f |^{p}   \rangle_{J}^{1/p}
    \langle | g |^{q} \rangle_{J}^{1/q} .
  \end{eqnarray*}
  This means that the function we are looking for takes the form:
  \begin{eqnarray*}
    B_{p} ( \tmmathbf{F},\tmmathbf{G},\tmmathbf{f},\tmmathbf{g} ) & = &
    \sup_{f,g} \left\{ \frac{1}{4 | J |} \sum_{I \in \mathcal{D} ( J )} | I |
    | \langle f \rangle_{I_{+}} - \langle f \rangle_{I_{-}} |   | \langle g
    \rangle_{I_{+}} - \langle g \rangle_{I_{-}} | : \right.\\
  &&\left. \vphantom{ \frac{1}{4 | J |} \sum_{I \in \mathcal{D} ( J )}} \langle f \rangle_{J} =\tmmathbf{f}, \langle g \rangle_{J}
    =\tmmathbf{g}, \langle | f |^{p} \rangle_{J} =\tmmathbf{F}, \langle | g
    |^{q} \rangle_{J} =\tmmathbf{G} \right\} .
  \end{eqnarray*}
  The supremum runs over functions supported in $J$, but scaling shows that
  the function $B_{p}$ itself does not depend upon the interval $J$. The
  function $B_{p}$ has the following properties stated in the lemma:

  \begin{enumeratenumeric}
    \item We have $B_{p} \geqslant 0.$ This means that for any $(
    \tmmathbf{F},\tmmathbf{G},\tmmathbf{f},\tmmathbf{g} ) \in D_{p}$, one can
    find complex valued functions $f$ and $g$ such that $( \langle f
    \rangle_{J} , \langle g \rangle_{J} , \langle | f |^{p} \rangle_{J} ,
    \langle | g |^{q} \rangle_{J} ) = (
    \tmmathbf{F},\tmmathbf{G},\tmmathbf{f},\tmmathbf{g} )$. Otherwise the
    supremum would run over the empty set yielding $B_{p} (
    \tmmathbf{F},\tmmathbf{G},\tmmathbf{f},\tmmathbf{g} ) =- \infty$. On the
    contrary, as soon as there exist $( f,g )$ such that $( \langle f
    \rangle_{J} , \langle g \rangle_{J} , \langle | f |^{p} \rangle_{J} ,
    \langle | g |^{q} \rangle_{J} ) = (
    \tmmathbf{F},\tmmathbf{G},\tmmathbf{f},\tmmathbf{g} )$, the expression of
    $B_{p}$ ensures that $B_{p} (
    \tmmathbf{F},\tmmathbf{G},\tmmathbf{f},\tmmathbf{g} ) \geqslant 0$.
    
    Let therefore $( \tmmathbf{F},\tmmathbf{G},\tmmathbf{f},\tmmathbf{g} ) \in
    D_{p}$. Choose $f$ so that $| f | =\tmmathbf{F}^{1/p} .$ Then trivially
    $\langle | f |^{p} \rangle_{J} =\tmmathbf{F}.$ Choose $\varphi$ so that
    $e^{i \varphi} \tmmathbf{f} \in \mathbbm{R}$. Since $| \tmmathbf{f} |
    \leqslant \tmmathbf{F}^{1/p}$ there is $c \in \mathbbm{R}:  | c |
    \leqslant 1$ with $e^{i \varphi} \tmmathbf{f}=c\tmmathbf{F}^{1/p}$. Now
    choose $K $a subinterval of $J$ so that $c= \frac{- | K | + | J \backslash
    K |}{| J |}$ and choose $f \equiv -e^{-i \varphi} \tmmathbf{F}^{1/p}  $on
    $K$ and $f \equiv e^{-i \varphi} \tmmathbf{F}^{1/p}$ on $J \backslash K$.
    Similar considerations for variables $\tmmathbf{g},\tmmathbf{G}$ show that
    $B_{p}  >- \infty$ and therefore $B_{p} \geqslant 0.$
    
    \item The upper estimate on $B_{p}  $ follows from Burkholder's theorem
    and the fact that we only allow values $(
    \tmmathbf{F},\tmmathbf{G},\tmmathbf{f},\tmmathbf{g} )$ that can occur when
    these numbers are averages of functions $f,g$ respectively. This is why
    the domain is restricted so as not to violate H{\"o}lder's inequality: $|
    \tmmathbf{f} |^{p} \leqslant \tmmathbf{F}, | \tmmathbf{g} |^{q} \leqslant
    \tmmathbf{G}$.
    
    \item The dissipation estimate can be seen as follows. Fix the interval
    $J$ and let $2\tmmathbf{v}=\tmmathbf{v}_{+} +\tmmathbf{v}_{-}$ with
    $\tmmathbf{v},\tmmathbf{v}_{+} ,\tmmathbf{v}_{-}$ in the domain of $B$.
    Now construct functions $f,g$ with the prescribed averages on $J_{+}$ and
    $J_{-}$ respectively as shown above. The resulting functions $f,g$ defined
    on $J$ correspond to $\tmmathbf{v}.$ So
    \begin{eqnarray*}
      B ( \tmmathbf{v} ) & \geqslant & \frac{1}{4 | J |} \sum_{I \in
      \mathcal{D} ( J )} | I | | \langle f \rangle_{I_{+}} - \langle f
      \rangle_{I_{-}} |   | \langle g \rangle_{I_{+}} - \langle g
      \rangle_{I_{-}} |\\
      & \geqslant & \frac{1}{4} | \tmmathbf{f}_{+} -\tmmathbf{f}_{-} |   |
      \tmmathbf{g}_{+} -\tmmathbf{g}_{-} | \\
      & & + \frac{1}{4 | J |} \sum_{I \in
      \mathcal{D} ( J ) ,I \neq J} | I | | \langle f \rangle_{I_{+}} - \langle
      f \rangle_{I_{-}} |   | \langle g \rangle_{I_{+}} - \langle g
      \rangle_{I_{-}} | .
    \end{eqnarray*}
    Taking supremum over all $f,g$ made to match $\tmmathbf{v}_{\pm}$ gives
    the required convexity property (\ref{eq: dissipation dyadic}) of $B_{p}$.
  \end{enumeratenumeric}
  This concludes the proof of Lemma \ref{L: Bellman function dyadic}.
\end{proof}

Notice that the function $B_{p}$ attained in this manner need not be smooth.
This is the point of Theorem \ref{T: Bellman function with smoothing}. We
refer the reader also to {\cite{VolNaz2004a}}, where this has previously been
done.

\begin{proof}
  (of Theorem \ref{T: Bellman function with smoothing}) Using Lemma \ref{L:
  Bellman function dyadic} above, the rest of the proof is a standard
  mollifying argument, combined with the observation that midpoint concavity
  is the same as concavity defined through the Hessian. This is a well known
  fact that only uses an application of Taylor's theorem and integration
  against a tent function (Green's function on $\mathbbm{Z}$). Let us fix a
  compact, convex set $K \subset D_{p}$ and a corresponding $\varepsilon >0$,
  that is very small in comparison to the distance of $K$ to the boundary of
  the domain $D_{p}$. Take a standard mollifier $\varphi$ \ and mollify
  $B_{p}$ by convolution with dilates of $\varphi$. Then one can show that the
  resulting function $B_{\varepsilon ,p,K}$ has the properties
  \begin{eqnarray*}
    0 \leqslant B_{\varepsilon ,p,K} ( \tmmathbf{v} ) \leqslant ( 1+
    \varepsilon C_{K} ) ( p^{\ast} -1 ) \tmmathbf{F}^{1/p} \tmmathbf{G}^{1/q},
  \end{eqnarray*}
  and
  \begin{eqnarray}
    - \mathd^{2}_{\tmmathbf{v}} B_{\varepsilon ,p,K} ( \tmmathbf{v} )
    \geqslant 2 | \mathd \tmmathbf{f} |   | \mathd \tmmathbf{g} | . \label{eq:
    Hessian estimate of B}
  \end{eqnarray}
  
\end{proof}

Now, we are ready to derive our dynamics condition in the lemma below. In
continuous settings, this becomes just an application of the chain rule. The
absence of a chain rule in the discrete setting is a true obstacle that is
very characteristic of discrete groups. \ It can be overcome in this very
particular case, because our Hessian estimate (\ref{eq: Hessian estimate of
B}) for the function $B ( \tmmathbf{v} )$ is universal and does not depend on
$\tmmathbf{v}$.

\begin{lemma}
  \label{L: dissipation carre du champ}Let us assume a compact and convex
  subset $K \subset D_{p}$ has been chosen. Let
  \begin{eqnarray*}
    \tilde{\tmmathbf{v}} ( n,t ) \assign ( \widetilde{| f |^{p}} ( n,t ) ,
    \widetilde{| g |^{q}} ( n,t ) , \tilde{f} ( n,t ) , \tilde{g} ( n,t ) )
  \end{eqnarray*}
  and assume that $\tilde{\tmmathbf{v}} ( n,t )$ and its neighbors
  $\tilde{\tmmathbf{v}} ( n \pm e_{i} ,t )$ lie in the domain $K$ of
  $B_{\varepsilon ,p,K} .$ Then
  \begin{eqnarray}
    \begin{array}{lll}
      ( \partial_{t} - \Delta ) ( B_{\varepsilon ,p,K} \circ  
      \tilde{\tmmathbf{v}} ) ( n,t ) & \geqslant & \sum_{i} ( |
      \partial^{i}_{+} \tilde{f} |   | \partial^{i}_{+} \tilde{g} | + |
      \partial^{i}_{-} \tilde{f} |   | \partial^{i}_{-} \tilde{g} | ) .
    \end{array} \label{eq: dissipation carre du champ}
  \end{eqnarray}
\end{lemma}

\begin{proof}
  We will write $B=B_{\varepsilon ,p,K}$. Taylor's theorem and the smoothness
  of $B$ allow us to write,
  \begin{eqnarray*}
    B ( \tmmathbf{v}+ \delta \tmmathbf{v} ) =B ( \tmmathbf{v} ) +
    \nabla_{\tmmathbf{v}} B ( \tmmathbf{v} ) \delta \tmmathbf{v}+ \int^{1}_{0}
    ( 1-s ) ( \mathd_{\tmmathbf{v}}^{2} B ( \tmmathbf{v}+s \delta \tmmathbf{v}
    ) \delta \tmmathbf{v}, \delta \tmmathbf{v} )   \mathd s.
  \end{eqnarray*}
  In particular, for any $1 \leqslant i \leqslant N$ we have for fixed $( n,t
  )$ with $\delta^{\pm}_{i} \tilde{\tmmathbf{v}}$ such that
  $\tilde{\tmmathbf{v}} + \delta^{\pm}_{i} \tilde{\tmmathbf{v}} =
  \tilde{\tmmathbf{v}} ( \cdot \pm e_{i} , \cdot )$
  \begin{eqnarray*}
    B ( \tilde{\tmmathbf{v}} + \delta_{i}^{\pm} \tilde{\tmmathbf{v}} )& = &B (
    \tilde{\tmmathbf{v}} ) + \nabla_{\tmmathbf{v}} B ( \tilde{\tmmathbf{v}} )
    \delta_{i}^{\pm} \tilde{\tmmathbf{v}} \\
    &&+ \int^{1}_{0} ( 1-s ) (
    \mathd_{\tmmathbf{v}}^{2} B ( \tilde{\tmmathbf{v}} +s \delta_{i}^{\pm}
    \tilde{\tmmathbf{v}} ) \delta_{i}^{\pm} \tilde{\tmmathbf{v}} ,
    \delta_{i}^{\pm} \tilde{\tmmathbf{v}} )   \mathd s.
  \end{eqnarray*}
  and summing over all $i$ and all $\pm$'s, we have
  \begin{eqnarray*}
    \Delta ( B \circ \tilde{\tmmathbf{v}} ) & = & \sum_{i, \pm} B (
    \tilde{\tmmathbf{v}} + \delta_{i}^{\pm} \tilde{\tmmathbf{v}} ) -B (
    \tilde{\tmmathbf{v}} )\\
    & = & \sum_{i, \pm} \nabla_{\tmmathbf{v}} B ( \tilde{\tmmathbf{v}} )
    \delta_{i}^{\pm} \tilde{\tmmathbf{v}} + \sum_{i, \pm} \int^{1}_{0} ( 1-s )
    ( \mathd_{\tmmathbf{v}}^{2} B ( \tilde{\tmmathbf{v}} +s \delta_{i}^{\pm}
    \tilde{\tmmathbf{v}} ) \delta_{i}^{\pm} \tilde{\tmmathbf{v}} ,
    \delta_{i}^{\pm} \tilde{\tmmathbf{v}} )   \mathd s\\
    & = & \nabla_{\tmmathbf{v}} B ( \tilde{\tmmathbf{v}} ) \Delta_{}
    \tilde{\tmmathbf{v}} + \sum_{i, \pm} \int^{1}_{0} ( 1-s ) (
    \mathd_{\tmmathbf{v}}^{2} B ( \tilde{\tmmathbf{v}} +s \delta_{i}^{\pm}
    \tilde{\tmmathbf{v}} ) \delta_{i}^{\pm} \tilde{\tmmathbf{v}} ,
    \delta_{i}^{\pm} \tilde{\tmmathbf{v}} )   \mathd s
  \end{eqnarray*}
  On the other hand, the boundedness of $\tilde{\tmmathbf{v}}$ together with
  the fact that the discrete laplacian is a bounded operator from $L^{\infty}
  ( G^{N} )$ into itself easily imply that $\tilde{\tmmathbf{v}}$ is
  continuously differentiable w.r.t. the variable $t$, and we have,
  \begin{eqnarray*}
    \frac{\partial}{\partial t} ( B \circ \tilde{\tmmathbf{v}} ) =
    \nabla_{\tmmathbf{v}} B ( \tilde{\tmmathbf{v}} ) \frac{\partial
    \tilde{\tmmathbf{v}}}{\partial t} .
  \end{eqnarray*}
  Now $\tilde{\tmmathbf{v}} ( n,t )$ is a vector consisting of solutions to
  the semi-continuous heat equation. This means we have $( \partial_{t} -
  \Delta ) \tilde{\tmmathbf{v}} =0$, that is
  \begin{eqnarray*}
    \partial_{t} \tilde{\tmmathbf{v}} ( n,t ) - \sum_{i} (
    \tilde{\tmmathbf{v}} ( n+e_{i} ,t ) -2 \tilde{\tmmathbf{v}} ( n,t ) +
    \tilde{\tmmathbf{v}} ( n-e_{i} ,t ) ) =0.
  \end{eqnarray*}
  The difference of the last two equations reads
  \begin{eqnarray*}
    \left( \frac{\partial}{\partial t} - \Delta \right) ( B \circ
    \tilde{\tmmathbf{v}} ) & = & \nabla_{\tmmathbf{v}} B (
    \tilde{\tmmathbf{v}} ) ( \partial_{t} - \Delta ) \tilde{\tmmathbf{v}}\\
    &  & + \sum_{i, \pm} \int^{1}_{0} ( 1-s ) ( - \mathd_{\tmmathbf{v}}^{2} B
    ( \tilde{\tmmathbf{v}} +s \delta_{i}^{\pm} \tilde{\tmmathbf{v}} )
    \delta_{i}^{\pm} \tilde{\tmmathbf{v}} , \delta_{i}^{\pm}
    \tilde{\tmmathbf{v}} )   \mathd s.
  \end{eqnarray*}
  Notice that the domain of $B$ is convex. Recall that $( n,t )$ is such that
  both $\tilde{\tmmathbf{v}}$ and $\tilde{\tmmathbf{v}} + \delta^{i}_{\pm}
  \tilde{\tmmathbf{v}} = \tilde{\tmmathbf{v}} ( \cdot \pm e_{i} , \cdot )$ at
  $( n,t )$ are contained in the domain of $B$. Thanks to convexity, so is
  $\tilde{\tmmathbf{v}} +s \delta^{i}_{+} \tilde{\tmmathbf{v}}$ for $0
  \leqslant s \leqslant 1$. Together with the universal inequality $-
  \mathd_{\tmmathbf{v}}^{2} B ( \tmmathbf{v} ) \geqslant 2 | \mathd
  \tmmathbf{f} |   | \mathd \tmmathbf{g} |$ in the domain of $B \nocomma$, we
  finally obtain the estimate
  \begin{eqnarray*}
    ( \partial_{t} - \Delta ) ( B \circ \tilde{\tmmathbf{v}} ) & \geqslant &
    \sum_{i} ( | \partial^{i}_{+} \tilde{f} |   | \partial^{i}_{+} \tilde{g} |
    + | \partial^{i}_{-} \tilde{f} |   | \partial^{i}_{-} \tilde{g} | ).
  \end{eqnarray*}
  after observing that $\delta_{\pm} = \pm \partial_{\pm}$.
\end{proof}

\subsection{Bilinear embedding}

We are going to prove the following stated as Theorem \ref{T: bilinear
embedding p minus 1} in the Introduction:

\begin{theorem}
Let $f$ and $g$ be test functions on $G^{N}$ and $p$ and $q$
conjugate exponents. Then we have the estimate
\begin{eqnarray*}
2  \sum_{i=1}^{N} \int^{\infty}_{0} \sum_{G^{N}} | \partial^{i}_{+} \tilde{f} (
  n,t ) |   | \partial_{+}^{i} \tilde{g} ( n,t ) |  \mathd  t
  \leqslant ( p^{\ast} -1 ) \| f \|_{p} \| g \|_{q} .
\end{eqnarray*}
\end{theorem}

\begin{proof}
  We write for simplicity $Q=G^{N}$. Notice first that for any $1 \leqslant p
  \leqslant \infty$, the discrete character of $Q^{}$ ensures that $L^{p} ( Q
  ) \subset L^{\infty} ( Q )$, namely $\| f \|_{\infty} \leqslant \| f
  \|_{p}$. From the elementary properties (\ref{eq: elementary properties heat
  kernel}) of the heat kernel, the heat semi-group is a contraction in
  $L^{\infty} ( Q )$. As a conclusion, for all $1<p< \infty ,$ for all $f \in
  L^{p} ( Q )$, we have for all $t \geqslant 0$, $\| \tilde{f} ( t )
  \|_{\infty} \leqslant \| f \|_{\infty} \leqslant \| f \|_{p}$. Similar
  observations hold for the functions $\tilde{g}$, $\widetilde{| f |^{p}}$ and
  $\widetilde{| g |^{q}}$. Setting as in Lemma \ref{L: dissipation carre du
  champ}
  \begin{eqnarray*}
    \tilde{\tmmathbf{v}} ( n,t ) \assign ( \widetilde{| f |^{p}} ( n,t ) ,
    \widetilde{| g |^{q}} ( n,t ) , \tilde{f} ( n,t ) , \tilde{g} ( n,t ) ) ,
  \end{eqnarray*}
  we have that $\tilde{\tmmathbf{v}} ( Q \times [ 0, \infty ) )$ lies in a
  compact $K \subset D_{p}$, the domain of the Bellman function. Pick a point
  $( n,t ) \in Q \times ( 0, \infty )$ and set $b ( n,t ) =B_{\varepsilon
  ,p,K} \circ \tilde{\tmmathbf{v}} ( n,t )$. Invoking Theorem \ref{T: Bellman
  function with smoothing}, we have therefore on the one hand the upper
  estimate $ \forall ( n,t ) $
  \begin{eqnarray*}
    \hspace{1em} ( B \circ \tilde{\tmmathbf{v}} ) ( n,t )
    \assign b ( n,t ) \leqslant ( 1+ \varepsilon C_{K} ) ( p^{\ast} -1 ) (
    \widetilde{| f |^{p}} ( n,t ) )^{1/p} ( \widetilde{| g |^{q}} ( n,t )
    )^{1/q} ,
  \end{eqnarray*}
  or in a condensed manner as an inequality involving functions,
  \begin{eqnarray*}
    \forall t, \hspace{1em} ( B \circ \tilde{\tmmathbf{v}} ) ( t ) \assign b (
    t ) \leqslant ( 1+ \varepsilon C_{K} ) ( p^{\ast} -1 ) ( \widetilde{| f
    |^{p}} ( t ) )^{1/p} ( \widetilde{| g |^{q}} ( t ) )^{1/q} .
  \end{eqnarray*}
  On the other hand, for a lower estimate, we wish to make explicit the
  dependence of $b ( t )$ with respect to the dissipation estimate (\ref{eq:
  dissipation carre du champ}). Introduce now the function $\mathcal{D}$
  defined on $Q \times [ 0,t ]$ as
  \begin{eqnarray*}
    \forall 0 \leqslant s \leqslant t, \hspace{1em} \mathcal{D} ( s ) \assign
    b ( s ) \ast K ( t-s )
  \end{eqnarray*}
  so that $\mathcal{D} ( t ) =b ( t ) = ( B \circ \tmmathbf{v} ) ( t )$. We
  have, using standard discrete integration by parts that are allowed thanks
  to the exponential decay of the kernel,
  \begin{eqnarray*}
    \mathcal{D}' ( s ) & = & ( \partial_{t} b ) ( s ) \ast K ( t-s ) -b ( s )
    \ast ( \partial_{t} K ) ( t-s )\\[.5em]
    & = & ( \partial_{t} b ) ( s ) \ast K ( t-s ) -b ( s ) \ast ( \Delta K )
    ( t-s )\\[.5em]
    & = & ( ( \partial_{t} - \Delta ) b ) ( s ) \ast K ( t-s ) ,
  \end{eqnarray*}
  from which follows
  \begin{eqnarray*}
    b ( t ) & = & b ( 0 ) \ast K ( t ) + \int_{0}^{t} ( ( \partial_{t} -
    \Delta ) b ) ( s ) \ast K ( t-s ) \mathd s\\
    & \geqslant & \int_{0}^{t} \Gamma ( \tilde{f} , \tilde{g} ) ( s ) \ast K
    ( t-s ) \mathd s.
  \end{eqnarray*}
  We used that $K \geqslant 0^{}$, $b \geqslant 0$, and the dissipation
  estimate (\ref{eq: dissipation carre du champ}) with $\Gamma ( \tilde{f} ,
  \tilde{g} )$ as a shorthand for:
  \begin{eqnarray*}
    \Gamma ( \tilde{f} , \tilde{g} ) := \sum_{i=1}^{N} | \partial^{i}_{+}
    \tilde{f} |   | \partial_{+}^{i} \tilde{g} | + | \partial^{i}_{-}
    \tilde{f} |   | \partial_{-}^{i} \tilde{g} |
  \end{eqnarray*}
  Summarizing, we have compared the functions $\forall t,$
  \begin{eqnarray*}
     \hspace{1em} \int_{0}^{t} \Gamma ( \tilde{f} , \tilde{g} ) ( s
    ) \ast K ( t-s ) \mathd s &\leqslant& b ( t ) \\
    &\leqslant &( 1+ \varepsilon
    C_{K} ) ( p^{\ast} -1 ) ( \widetilde{| f |^{p}} ( t ) )^{1/p} (
    \widetilde{| g |^{q}} ( t ) )^{1/q}.
  \end{eqnarray*}
  H{\"o}lder's inequality ensures that the right hand side is integrable over
  $Q \assign G^{N}$ uniformly in $t$. $ \forall t, $
  \begin{eqnarray*}
   \hspace{1em} \sum_{n \in Q} ( \widetilde{| f |^{p}} ( n,t )
    )^{1/p} ( \widetilde{| g |^{q}} ( n,t ) )^{1/q} 
    & \leqslant &
     \left(
    \sum_{n \in Q} \widetilde{| f |^{p}} ( n,t ) \right)^{1/p}  \left(
    \sum_{n \in Q} \widetilde{| g |^{q}} ( n,t ) \right)^{1/q}\\[.1em]
    & \leqslant & ( \| \widetilde{| f |^{p}} ( t ) \|_{1} )^{1/p}  ( \|
    \widetilde{| g |^{q}} ( t ) \|_{1} )^{1/q}\\[.5em]
    & \leqslant & \left( \left\| \; | f |^{p} \; \right\|_{1} \right)^{1/p}
     \left( \left\lVert \; | g |^{q} \; \right\rVert_{1} \right)^{1/q}\\[.5em]
    & \leqslant & \| f \|_{p} \| g \|_{q} < \infty ,
  \end{eqnarray*}
  where we used that the heat semi-group is a contraction in $L^{1}$ (actually
  preserves norms of nonnegative functions). For the left hand side,
  integrating also on $Q \assign G^{N}$ and using Fubini for absolutely
  converging series, we have
  \begin{eqnarray*}
  \lefteqn{\sum_{n \in Q} \int_{0}^{t} ( \Gamma ( \tilde{f} , \tilde{g} ) ( s ) \ast
    K ( t-s ) ) ( n ) \mathd s }\\
     & \qquad \qquad \qquad=&  \sum_{n \in Q} \int_{0}^{t} \sum_{m \in
    Q} \Gamma ( \tilde{f} , \tilde{g} ) ( m,s ) K ( n-m,t-s ) \mathd s\\
    &\qquad \qquad \qquad=&  \int_{0}^{t} \sum_{m \in Q} \Gamma ( \tilde{f} , \tilde{g} ) ( m,s
    ) \left( \sum_{n \in Q} K ( n-m,t-s ) \right) \mathd s\\
   & \qquad \qquad \qquad= & \int_{0}^{t} \sum_{m \in Q} \Gamma ( \tilde{f} , \tilde{g} ) ( m,s
    ) \mathd s
  \end{eqnarray*}
  Letting finally $\varepsilon$ go to $0$ and $t$ go to infinity yields the
  result.
\end{proof}

\subsection{Proof of Theorem \ref{T: p minus 1 estimate}}

The identity formulae for second order Riesz transforms allow us to deduce the
estimate from the bilinear embedding.

\begin{proof}
  (of Theorem \ref{T: p minus 1 estimate}) Remember that
  \begin{eqnarray*}
    R_{\alpha}^{2} \assign \sum^{N}_{i=1} \alpha_{i} R^{2}_{i} =
    \sum^{N}_{i=1} \alpha_{i} R^{i}_{+} R^{i}_{-} , \hspace{2em} \alpha_{i}
    \in \mathbbm{C}, \hspace{1em} | \alpha_{i} | \leqslant 1, \hspace{1em}
    \forall i \in [ 1, \ldots ,N ] .
  \end{eqnarray*}
  Using successively the representation formula (\ref{eq: representation
  formula}) of Lemma \ref{L: representation formula} and the bilinear estimate
  (\ref{eq: bilinear embedding p minus 1}) of Theorem \ref{T: bilinear
  embedding p minus 1}, we have

  \begin{eqnarray*}
    | ( f,R_{\alpha}^{2} g ) | & \assign & \left\lvert \sum_{i} \alpha_{i}
    \int^{\infty}_{0} \sum_{G^{N}} \partial^{i}_{+} \tilde{f} ( n,t )
    \overline{\partial_{+}^{i} \tilde{g} ( n,t )} + \partial^{i}_{-} \tilde{f}
    ( n,t ) \overline{\partial_{-}^{i} \tilde{g} ( n,t )} \mathd t \right\lvert\\
    & \leqslant & \sum_{i} | \alpha_{i} | \int^{\infty}_{0} \sum_{G^{N}} |
    \partial^{i}_{+} \tilde{f} ( n,t ) |   | \partial_{+}^{i} \tilde{g} ( n,t
    ) | + | \partial^{i}_{-} \tilde{f} ( n,t ) |   | \partial_{-}^{i}
    \tilde{g} ( n,t ) | \mathd t\\
    & \leqslant & ( p^{\ast} -1 ) \| f \|_{p} \| g \|_{q}
  \end{eqnarray*}
  which proves the result.
\end{proof}

\section{The Choi constant}

For some multipliers, the best constant is better than $p^{\ast} -1$. The
simplest example is when $\alpha \equiv 1$, in which case we estimate the
identity. When considered as an operator on real valued functions only, the
constant $p^{\ast} -1$ is not necessarily best possible also for some
non-trivial cases. The estimates of multipliers with real non-negative
coefficients will depend upon the relative distance of the largest and the
smallest multiplier. Instead of Burkholder's result, we use a theorem by Choi
{\cite{Cho1992a}}.

\begin{theorem}
  (Choi) Let $1<p< \infty$ and let $d_{k}$ be a real valued martingale
  difference sequence and let $\theta_{k}$ be a predictable sequence taking
  values in $\{ 0,1 \}$. The best constant in the inequality
  \begin{eqnarray*}
    \left\lVert \sum^{n}_{k=1} \theta_{k} \mathd  _{k} \right\rVert_{p} \leqslant
    \mathfrak{C}_{p} \left\| \sum^{n}_{k=1} \mathd_{k} \right\|_{p}
  \end{eqnarray*}
  that holds for all $n$ is the Choi constant $\mathfrak{C}_{p}$.
\end{theorem}

It is a stronger hypothesis than differential subordination. We transform this
martingale estimate into the existence of an appropriate Bellman function. We
prove now

\begin{lemma}
  \label{L: Choi Bellman function dyadic}For every $1<p< \infty$ there exists
  a function $C^{+}_{p}$ of four variables $\tmmathbf{v}= (
  \tmmathbf{F},\tmmathbf{G},\tmmathbf{f},\tmmathbf{g} )$ in the domain
  $\overline{D_{p}} = \{ ( \tmmathbf{F},\tmmathbf{G},\tmmathbf{f},\tmmathbf{g}
  ) \in \mathbbm{R}^{+} \times \mathbbm{R}^{+} \times \mathbbm{R} \times
  \mathbbm{R}: | \tmmathbf{f} |^{p} \leqslant \tmmathbf{F}, | \tmmathbf{g}
  |^{q} \leqslant \tmmathbf{G} \}$ so that
  \begin{eqnarray*}
    0 \leqslant C_{p}^{+} ( \tmmathbf{v} ) \leqslant \mathfrak{C}_{p}
    \tmmathbf{F}^{1/p} \tmmathbf{G}^{1/q} ,
  \end{eqnarray*}
  and
  \begin{eqnarray*}
    C_{p}^{+} ( \tmmathbf{v} ) \geqslant \frac{1}{2} C_{p}^{+} (
    \tmmathbf{v}_{+} ) + \frac{1}{2} C_{p}^{+} ( \tmmathbf{v}_{-} )   \noplus
    \noplus + \frac{1}{4} [ ( \tmmathbf{f}_{+} -\tmmathbf{f}_{-} ) (
    \tmmathbf{g}_{+} -\tmmathbf{g}_{-} ) ]_{+}
  \end{eqnarray*}
  if $\tmmathbf{v}_{+} +\tmmathbf{v}_{-} =2\tmmathbf{v}$ and $\tmmathbf{v}$,
  $\tmmathbf{v}_{+}$, $\tmmathbf{v}_{-}$ are in the domain. Here $\left[ \,
  \cdot \, \right]_{+} \assign \max \{ \cdot ,0 \}$ denotes non-negative part.
\end{lemma}

\begin{proof}
  In all the sequel, we consider the case where $\theta_{k}$ takes values in
  $\{ 0,1 \}$. It will be clear from the context how to handle the general
  case. Similar to above, this implies that the operator
  \begin{eqnarray*}
    T_{c} :f \mapsto \sum_{I \in \mathcal{D} ( J )} c_{I} ( f,h_{I} ) h_{I}
  \end{eqnarray*}
  where $c= \{ c_{I} \}_{I \in \mathcal{D}} \in \{ 0,1 \}^{\mathcal{D}}$, has
  $L^{p}$ bound, uniformly in the choice of the sequence $c$, and has norm at
  most $\mathfrak{C}_{p}$. By duality this becomes
  \begin{eqnarray*}
    \sup_{c} | ( T_{c} f,g ) | \leqslant \mathfrak{C}_{p} \| f \|_{p} \| g
    \|_{q} .
  \end{eqnarray*}
  Looking at the left hand side, the supremum is attained, for a given pair of
  test functions $f,g$ when one omits either the positive or the negative
  terms of the sum. That is either the sequence
  \begin{eqnarray*}
    c= \left\{ c_{I} : \hspace{1em} c_{I} =1 \hspace{1em} \tmop{for}
    \hspace{1em} ( f,h_{I} ) ( g,h_{I} ) \geqslant 0 \nocomma , \hspace{1em}
    c_{I} =0 \hspace{1em} \tmop{otherwise} \right\}
  \end{eqnarray*}
  or the sequence
  \begin{eqnarray*}
    c= \left\{ c_{I} : \hspace{1em} c_{I} =1 \hspace{1em} \tmop{for}
    \hspace{1em} ( f,h_{I} ) ( g,h_{I} ) \leqslant 0 \nocomma , \hspace{1em}
    c_{I} =0 \hspace{1em} \tmop{otherwise} \right\}
  \end{eqnarray*}
  That is, the supremum writes as
  \begin{eqnarray*}
    | ( T_{c} f,g ) | = \sum_{I} [ ( f,h_{I} ) ( g,h_{I} ) ]_{+} \text{ or }  | ( T_{c} f,g ) | = \sum_{I} [ ( f,h_{I} ) ( g,h_{I} ) ]_{-}.
  \end{eqnarray*}
  Here, $\left[ \, \cdot \, \right]_{\pm}$ denotes the positive
  or negative part of the sequence. We deal now with the estimate of the
  positive part. Let us define the following function
  \begin{eqnarray*}
    C^{+}_{p} ( \tmmathbf{F},\tmmathbf{G},\tmmathbf{f},\tmmathbf{g} ) & = &
    \sup_{f,g} \left\{ \frac{1}{4 | J |} \sum_{I \in \mathcal{D} ( J )} | I |
    [ ( \langle f \rangle_{I_{+}} - \langle f \rangle_{I_{-}} )   ( \langle g
    \rangle_{I_{+}} - \langle g \rangle_{I_{-}} ) ]_{+} : \right.\\
     && \qquad \left. \vphantom{\frac{1}{4 | J |} \sum_{I \in \mathcal{D} ( J )} } \langle f \rangle_{J} =\tmmathbf{f}, \langle g \rangle_{J}
    =\tmmathbf{g}, \langle | f |^{p} \rangle_{J} =\tmmathbf{F}, \langle | g
    |^{q} \rangle_{J} =\tmmathbf{G} \right\} .
  \end{eqnarray*}
  The supremum runs over functions supported in $J \nocomma \nocomma \nocomma
  ,$but scaling shows that the function $C^{+}_{p}$ itself does not depend
  upon the interval $J.$ The domain of $C^{+}_{p}  $ is chosen so that every
  set of numbers $\tmmathbf{v}= (
  \tmmathbf{F},\tmmathbf{G},\tmmathbf{f},\tmmathbf{g} )$ with
  $\tmmathbf{F},\tmmathbf{G} \geqslant 0 \nocomma , | \tmmathbf{f} |^{p}
  \leqslant \tmmathbf{F}, | \tmmathbf{g} |^{q} \leqslant \tmmathbf{G}$ can be
  associated with real valued functions $f,g$ so that
  \begin{eqnarray*}
    \langle f \rangle_{J} =\tmmathbf{f}, \langle g \rangle_{J} =\tmmathbf{g},
    \langle | f |^{p} \rangle_{J} =\tmmathbf{F}, \langle | g |^{q} \rangle_{J}
    =\tmmathbf{G} \nosymbol .
  \end{eqnarray*}
  As discussed before, we know that 
  \begin{eqnarray*}
    0 \leqslant C^{+}_{p} ( \tmmathbf{v} ) \leqslant \mathfrak{C}_{p}
    \tmmathbf{F}^{1/p} \tmmathbf{G}^{1/q} .
  \end{eqnarray*}
  We will now investigate the contributions of $J_{\pm}$ to the above sum and
  derive a concavity condition for $C^{+}_{p}$. Choose as above
  $\tmmathbf{v}_{\pm}$ in the domain so that $\tmmathbf{v}_{+}
  +\tmmathbf{v}_{-} =2\tmmathbf{v}$. Then
  \begin{eqnarray*}
    \lefteqn{C^{+}_{p} ( \tmmathbf{F},\tmmathbf{G},\tmmathbf{f},\tmmathbf{g} ) }\\
    & \qquad \geqslant & \sup_{f,g} \left\{ \frac{1}{4 | J |} \sum_{I \in \mathcal{D} (
    J )} | I | [ ( \langle f \rangle_{I_{+}} - \langle f \rangle_{I_{-}} )   (
    \langle g \rangle_{I_{+}} - \langle g \rangle_{I_{-}} ) ]_{+} :\right.\\
    &&\qquad \left. \vphantom{\frac{1}{8 | J |}\sum_{I \in \mathcal{D}(J_+)}}
      \langle f \rangle_{J_{\pm}} =\tmmathbf{f}_{\pm} , \langle g
    \rangle_{J_{\pm}} =\tmmathbf{g}_{\pm} , \langle | f |^{p}
    \rangle_{J_{\pm}} =\tmmathbf{F}_{\pm} , \langle | g |^{q} \rangle_{J \pm}
    =\tmmathbf{G}_{\pm} \right\}\\
    &\qquad \geqslant & \sup_{f,g} \left\{ \frac{1}{8 | J |} \sum_{I \in
    \mathcal{D} ( J_{+} )} | I | [ ( \langle f \rangle_{I_{+}} - \langle f
    \rangle_{I_{-}} )   ( \langle g \rangle_{I_{+}} - \langle g
    \rangle_{I_{-}} ) ]_{+}\right.\\
     &&\qquad + \frac{1}{8 | J |} \sum_{I \in \mathcal{D} ( J_{-} )} | I |
    [ ( \langle f \rangle_{I_{+}} - \langle f \rangle_{I_{-}} )   ( \langle g
    \rangle_{I_{+}} - \langle g \rangle_{I_{-}} ) ]_{+}  \\
     &&\qquad + \frac{1}{4} [ ( \langle f \rangle_{I_{+}} - \langle f
    \rangle_{I_{-}} )   ( \langle g \rangle_{I_{+}} - \langle g
    \rangle_{I_{-}} ) ]_{+} \noplus : \\
    && \qquad \left.
    \vphantom{\frac{1}{8 | J |}\sum_{I \in \mathcal{D}(J_+)}} \langle f \rangle_{J_{\pm}} =\tmmathbf{f}_{\pm} , \langle g
    \rangle_{J_{\pm}} =\tmmathbf{g}_{\pm} , \langle | f |^{p}
    \rangle_{J_{\pm}} =\tmmathbf{F}_{\pm} , \langle | g |^{q} \rangle_{J \pm}
    =\tmmathbf{G}_{\pm} \right\}\\
    &\qquad \geqslant&  \frac{1}{2} C^{+}_{p} ( \tmmathbf{F}_{+} ,\tmmathbf{G}_{+}
    ,\tmmathbf{f}_{+} ,\tmmathbf{g}_{+} ) + \frac{1}{2} C^{+}_{p} (
    \tmmathbf{F}_{-} ,\tmmathbf{G}_{-} ,\tmmathbf{f}_{-} ,\tmmathbf{g}_{-} ) \\
    &&\qquad + \frac{1}{4} [ ( \tmmathbf{f}_{+} -\tmmathbf{f}_{-} ) ( \tmmathbf{g}_{+}
    -\tmmathbf{g}_{-} ) ]_{+} .
  \end{eqnarray*}
  This proves Lemma \ref{L: Choi Bellman function dyadic}.
\end{proof}

An infinitesimal version of Lemma \ref{L: Choi Bellman function dyadic},
omitting smoothing parameters in the name of the function and{\color{red} 
}following the same considerations as in the proof of the range and Hessian
estimate (\ref{eq: Hessian estimate of B}) of Theorem \ref{T: Bellman function
with smoothing}, we have
\begin{eqnarray*}
  0 \leqslant C_{p}^{+} ( \tmmathbf{v} ) \leqslant ( 1+ \varepsilon C_{K} )
  \mathfrak{C}_{p} \tmmathbf{F}^{1/p} \tmmathbf{G}^{1/q} ,
\end{eqnarray*}
and
\begin{eqnarray*}
  - \mathd_{\tmmathbf{v}}^{2} C_{p}^{+} ( \tmmathbf{v} ) \geqslant ( 1+
  \tmop{sign} ( \mathd \tmmathbf{f} \mathd \tmmathbf{g} ) ) \mathd
  \tmmathbf{f} \mathd \tmmathbf{g}.
\end{eqnarray*}
Again, by Taylor's theorem
\begin{eqnarray*}
   C_{p}^{+} ( \tmmathbf{v}+ \delta \tmmathbf{v} ) &=&C_{p}^{+} ( \tmmathbf{v} )
  + \nabla_{\tmmathbf{v}} C_{p}^{+} ( \tmmathbf{v} ) \delta \tmmathbf{v}  \medskip\\
 & & \qquad+ \int^{1}_{0} ( 1-s ) ( \mathd_{\tmmathbf{v}}^{2} C_{p}^{+} ( \tmmathbf{v}+s
  \delta \tmmathbf{v} ) \delta \tmmathbf{v}, \delta \tmmathbf{v} ) \mathd s.
\end{eqnarray*}
Define as before the vector $\tilde{\tmmathbf{v}} ( n,t ) \assign (
\widetilde{| f |^{p}} ( n,t ) , \widetilde{| g |^{q}} ( n,t ) , \tilde{f} (
n,t ) , \tilde{g} ( n,t ) )$. Letting for fixed $( n,t )$, $\delta^{i}_{\pm}
\tilde{\tmmathbf{v}} = \pm \partial^{i}_{\pm} \tilde{\tmmathbf{v}} =
\tilde{\tmmathbf{v}} ( \cdot \pm e_{i} , \cdot ) - \tilde{\tmmathbf{v}}$, we
have
\begin{eqnarray*}
  C_{p}^{+} ( \tilde{\tmmathbf{v}} + \delta^{i}_{\pm} \tilde{\tmmathbf{v}} )
  &=&C_{p}^{+} ( \tilde{\tmmathbf{v}} ) + \nabla_{\tmmathbf{v}} C_{p}^{+} (
  \tilde{\tmmathbf{v}} ) \cdot \delta^{i}_{\pm} \tilde{\tmmathbf{v}} \\
  && \qquad +
  \int^{1}_{0} ( 1-s ) ( \mathd_{\tmmathbf{v}}^{2} C_{p}^{+} (
  \tilde{\tmmathbf{v}} +s \delta^{i}_{\pm} \tilde{\tmmathbf{v}} )
  \delta^{i}_{\pm} \tilde{\tmmathbf{v}} , \delta^{i}_{\pm}
  \tilde{\tmmathbf{v}} ) \mathd s,
\end{eqnarray*}
as well as
\begin{eqnarray*}
  \partial_{t} C_{p}^{+} ( \tilde{\tmmathbf{v}} ( n,t ) ) =
  \nabla_{\tmmathbf{v}} C_{p}^{+} ( \tilde{\tmmathbf{v}} ( n,t ) )
  \partial_{t} \tilde{\tmmathbf{v}} ( n,t ) .
\end{eqnarray*}
Summing over $i$ and $\pm$ and using that $\tilde{\tmmathbf{v}}$ solves the
heat equation, we obtain
\begin{eqnarray*}
  ( \partial_{t} - \Delta ) C_{p}^{+} ( \tilde{\tmmathbf{v}} ( n,t ) )=
  \sum_{i,\pm} \int^{1}_{0} ( 1-s ) \{ ( \mathd_{\tmmathbf{v}}^{2} C_{p}^{+} (
  \tilde{\tmmathbf{v}} +s \delta^{i}_{\pm} \tilde{\tmmathbf{v}} ) \delta^{i}_{\pm}
  \tilde{\tmmathbf{v}} , \delta^{i}_{\pm} \tilde{\tmmathbf{v}} ) \} \mathd s.
\end{eqnarray*}
Notice that the domain of $C_{p}^{+}$ is convex. Both $\tilde{\tmmathbf{v}} (
n,t )$ and $\tilde{\tmmathbf{v}} + \delta^{i}_{\pm} \tilde{\tmmathbf{v}} =
\tilde{\tmmathbf{v}} ( \cdot \pm e_{i} , \cdot )$ are contained in the domain
of $C_{p}^{+}$, and so is $\tilde{\tmmathbf{v}} +s \delta^{i}_{+}
\tilde{\tmmathbf{v}}$ for $0 \leqslant s \leqslant 1$. Together with the
universal inequality $- \mathd_{\tmmathbf{v}}^{2} C_{p}^{+} ( \tmmathbf{v} )
\geqslant ( 1+ \tmop{sign} (   \mathd \tmmathbf{f} \mathd \tmmathbf{g} ) )
\mathd \tmmathbf{f} \mathd \tmmathbf{g}$ in the domain of $C_{p}^{+}
\nocomma$, we obtain the estimate
\begin{eqnarray}
  ( \partial_{t} - \Delta ) C_{p}^{+} ( \tilde{\tmmathbf{v}} ( n,t ) )
  \geqslant \sum_{i} \left[ \partial^{i}_{+} \tilde{f} \hspace{0.5em}
  \partial^{i}_{+} \tilde{g} \right]_{+} + \left[ \partial^{i}_{-} \tilde{f}
  \hspace{0.5em} \partial^{i}_{-} \tilde{g} \right]_{+} \label{eq: dissipation
  estimate Choi}
\end{eqnarray}
With the help of these considerations, it is now easy to prove:

\begin{proof}
  (of Theorem \ref{T: Choi constant estimate} and Theorem \ref{T: bilinear
  embedding Choi})
  
  By the identity formula, we see that for $c_{i} \in \{ 0,1 \}$,
  \begin{eqnarray*}
    \left\lvert \left( f, \sum_{i} c_{i} R_{i}^{2} g \right) \right\vert & = & 2 \left\lvert
    \int^{\infty}_{0} \sum_{\mathbbm{Z}^{N}} \sum_{i} c_{i} \partial^{i}_{+}
    \tilde{f} ( n,t ) \partial_{+}^{i} \tilde{g} ( n,t ) \mathd t \right\lvert\\
    & \leqslant & \max_\pm  \left\{ 2 \int^{\infty}_{0} \sum_{\mathbbm{Z}^{N}}
    \sum_{i} [ \delta^{i}_{+} \tilde{f} \delta^{i}_{+} \tilde{g} ]_{\pm} \mathd t  \right\}\\
    & \leqslant & \mathfrak{C}_{p} \| f \|_{p} \| g \|_{q} .
  \end{eqnarray*}
  We used successively the representation formula (\ref{eq: representation
  formula}) of Lemma \ref{L: representation formula} {\tmstrong{}}the bilinear
  embedding of Theorem \ref{T: bilinear embedding Choi} resulting from the
  dissipation estimate (\ref{eq: dissipation estimate Choi}) above for
  $C_{p}^{+}$ above as well as the similar bilinear embedding that results
  from $C_{p}^{-}$. This concludes the proof of Theorem \ref{T: Choi constant
  estimate}. 
\end{proof}

\section{Sharpness.}

The sharpness in $\mathbbm{Z}^{N}  $is inherited from the continuous case
$\mathbbm{R}^{N}$. Just consider the isomorphic groups $( t\mathbbm{Z} )^{N}$
for $0<t \leqslant 1$ in conjunction with the Lax equivalence theorem
{\cite{LaxRic1956a}}. By the same argument, sharpness for a {\tmem{uniform}}
estimate in $m$ for the cyclic case $( \mathbbm{Z}/m\mathbbm{Z} )^{N}$ is
inherited from that on the torus $\mathbbm{T}^{N}$.

\end{document}